\documentclass[12pt]{amsart}
\usepackage{latexsym,enumerate}
\usepackage{amssymb, xcolor,mathrsfs}
\usepackage{amsmath,amsthm,amsfonts,amssymb,latexsym, mathabx}
\usepackage[pagebackref]{hyperref}

\newtheorem{theorem}{Theorem}[section]
\newtheorem{lemma}[theorem]{Lemma}

\theoremstyle{definition}

\newcommand{\C}{\mathbb{C}}

\newcommand{\la}{\langle}
\newcommand{\ra}{\rangle}

\newcommand{\ol}{\overline}

\newcommand{\Mod}[1]{\ (\mathrm{mod}\ #1)}

\DeclareMathOperator{\Syl}{Syl}

\DeclareMathOperator{\Irr}{Irr}

\usepackage{hyperref}

\begin{document}
	
\title[Burnside's Normal $p$-complement Theorem]
{A Block-theoretic Proof of Burnside's Normal $p$-complement Theorem}

	\author[Christopher Herbig]{Christopher Herbig}
	\address{%
		Northern Illinois University\\
		Department of Mathematical Sciences\\
		Dekalb, IL 60115\\
		USA}
	\email{cherbig@niu.edu}


	
	\begin{abstract} In \cite[Theorem 6.7B]{DixPut}, the authors use the Main Theorems of Brauer to give a proof of Burnside's Normal $p$-complement Theorem. Unfortunately, the proof contains an error. We take this opportunity to give a proof along similar lines, circumventing the error by means of a well-known result on traces of totally positive cyclotomic integers.
	\end{abstract}
	
	\maketitle
	
	
	
	\section{Introduction}
	
	Burnside's Normal $p$-complement Theorem is typically obtained via transfer theory. In \cite[Theorem 6.7B]{DixPut}, the authors use the Second and Third Main Theorems of Brauer to give a block-theoretic proof. Unfortunately, the proof contains an error which appears to have gone unnoticed in the literature to date. We complete the proof by applying a result of Siegel from \cite[Theorem III]{Sieg}, which states that the average among the Galois conjugates of a totally positive algebraic integer $\alpha$ is at least $3/2$ if $\alpha \neq 1$. John G. Thompson famously used Siegel's theorem to show that an irreducible character of a finite group is either zero or a root of unity on at least a third of the group elements (cf. \cite[Exercise 3.15]{Is}). For cyclotomic integers, Cassels provides a simpler proof of Siegel's theorem in \cite[Lemma 2]{Cass} which will be suitable for our purposes. 
	
	We begin by stating Burnside's Normal $p$-complement Theorem.
	\begin{theorem}\label{T:1}
		\textup{(Burnside)} Let $G$ be a finite group. If $C_G(P) = N_G(P)$ for $P \in \Syl_{p}(G)$, then $G$ has a normal $p$-complement.
	\end{theorem}
	\noindent The argument in \cite{DixPut} proceeds by choosing an arbitrary, irreducible ordinary character $\zeta$ from the principal $p$-block of a minimal counterexample. It is then shown that such a $\zeta$ must satisfy the relation
	
	\begin{equation}\label{E:1}
		\dfrac{1}{|G|}\sum_{g \in G} |\zeta(g) - \zeta(1)1_G|^2 \;\; \geq \;\; \dfrac{1}{|P|}\sum_{g \in P} |\zeta(g) - \zeta(1)1_G|^2
	\end{equation}
	for $P \in \Syl_p(G)$. It is also shown that if equality holds above, then each $\zeta$ in the principal block must contain $O^p(G)$, which leads to a contradiction. Now, the left side of \eqref{E:1} is equal to $1 + \zeta(1)^2$. However, the right side is strictly less than $1 + \zeta(1)^2$ if the restriction $\zeta_P$ has the trivial character as a constituent. 
	
	Using Siegel's theorem, we show that if $\zeta$ has height zero, then $\zeta$ must take root of unity values on each nonidentity $p$-element of $G$, and from this, we obtain a contradiction.
	
	\section{Preliminary Lemmas}

	We introduce some notation. $G$ is taken to be a minimal counterexample with $P \in \Syl_p(G)$. The principal $p$-block will be denoted as $B_0$, and $\Irr(B_0)$ will denote the set of all irreducible ordinary characters in $B_0$. For any group $H$, the set $H - \{1\}$ will be denoted as $H^{\#}$, and the set of all $p$-regular elements of $H$ will be denoted as $H^\circ$. For any two $\C$-class functions $\psi$ and $\eta$ of $G$, we let $\langle \psi , \eta \rangle_G$ denote the usual inner product of characters. Otherwise, our notation follows that of \cite{Is}.
	
	Our first lemma is a result of Brauer (cf. \cite[Corollary 5]{Brau}).
	
	\begin{lemma}\label{2:1}
		Let $\psi$ be a class function of $G$ such that each constituent lies in the principal block of $G$, and let $z$ be a $p$-element such that $H := C_G(z)$ has a normal $p$-complement, then for each $p$-regular element of $H$, we have $\psi(zy) = \psi(z)$.
	\end{lemma}
	
	\begin{lemma}\label{2:2}
		Under the hypothesis of Theorem \ref{T:1}, $A^g = B$ implies $A = B$ for any $A,B \subseteq P$ and $g \in G$.
	\end{lemma}
	
	\begin{proof}
		This is a consequence of Burnside's Fusion Theorem.
	\end{proof}
	
	The following lemma is a mild generalization of \eqref{E:1}.
	\begin{lemma}\label{2:3}
		Assume that $C_G(z)$ has a normal $p$-complement for all nonidentity $p$-elements of $G$. If $\psi$ and $\eta$ are any class functions of $G$ such that $\psi(zy) = \psi(z)$ and $\eta(zy) = \eta(z)$ for each nonidentity $p$-element $z$ and $y \in C_G(z)^\circ$, then 
		
		\begin{equation}\label{E:2}
			|G|\langle \psi , \eta \rangle_G = \sum_{g \in G^\circ} \psi(g)\ol{\eta(g)} + |G:P|\sum_{z \in P^{\#}} \psi(z)\ol{\eta(z)}.
		\end{equation}		
		
	\end{lemma}
	
	\begin{proof}
		Let $X$ denote the set of all nonidentity $p$-elements. Using the fact that each element of $G$ may be uniquely expressed as $g = zy = yz$ for a $p$-element $z$ and a $p$-regular element $y$, we have that
		
		$$
		|G|\langle \psi , \eta \rangle_G =  \sum_{g \in G} \psi(g)\ol{\eta(g)} = \sum_{z \in \{1\} \cup X} \sum_{y \in C_G(z)^\circ} \psi(zy)\ol{\eta(zy)}.
		$$ 
		Since each $C_G(z)$ has a normal $p$-complement, we have that $|C_G(z)^\circ| = |C_G(z):P|$ for $z \in X$. By Lemma \ref{2:1}, the right side of the above equation may be rewritten like so:
		
		$$
		\sum_{g \in G^\circ} \psi(g)\ol{\eta(g)} + \sum_{z \in  X} |C_G(z):P| \psi(z)\ol{\eta(z)}.
		$$
		Now, we may reindex the right sum above over $P^{\#}$ for some $P \in \Syl_p(G)$  since by Lemma \ref{2:2}, $X$ is a union of conjugacy classes which has $P^{\#}$ as a set of representatives:
		
		$$
		\sum_{z \in  X} |C_G(z):P| \psi(z)\ol{\eta(z)} = \sum_{z \in P^{\#}} |G:C_G(z)||C_G(z):P| \psi(z)\eta(z^{-1})
		$$ $$
		= \sum_{z \in P^{\#}} |G:P| \psi(z)\eta(z^{-1}),
		$$
		which yields \eqref{E:2}.
	\end{proof}
	
	We will require a number-theoretic lemma as well.
	
	\begin{lemma}\label{2:4}
		Let $\sum_{i = 1}^n \epsilon_i$ be a sum of $p$-power complex roots of unity. If this sum is equal to zero, then $n$ is a multiple of $p$.
	\end{lemma}
	
	\begin{proof}
		The lemma follows from the fact that for a maximal ideal $I$ of the algebraic integers containing $p$, $\epsilon \equiv 1 \Mod{I}$ for each $p$-power root of unity $\epsilon$.
	\end{proof}
	
	We record two final lemmas on the normal structure of $G$.
	
	\begin{lemma}\label{2:5}
		$O^p(G) = G$. 
	\end{lemma}
	
	\begin{proof}
		Assume that $O^p(G)$ is proper, and let $P_0$ be a Sylow $p$-subgroup of $O^p(G)$. It follows from Lemma \ref{2:2} that the assumptions of Theorem \ref{T:1} hold for  $O^p(G)$ and $P_0$. By induction, $O^p(G)$ has a normal $p$-complement, but as this complement is characteristic in $O^p(G)$, it follows that $G$ itself has a normal $p$-complement, a contradiction.
	\end{proof}

	\begin{lemma}\label{2:6}
		Let $P \in \Syl_p(G)$ and $N \triangleleft\, G$. If $C_G(P) = N_G(P)$, then $$C_{G/N}(PN/N) = N_{G/N}(PN/N).$$ In particular, $p \nmid |Z(G)|$.
	\end{lemma}
	
	\begin{proof}
		The first statement is seen by showing that 	
		$$N_{G/N}(PN/N) \subseteq N_G(P)N/N = C_{G}(P)N/N \subseteq C_{G/N}(PN/N).$$ 
		Each containment except for the left is straightforward to verify. For this containment, let $xN \in N_{G/N}(PN/N)$. We have that $PN = (PN)^{xN} = P^xN$. We have then that $P, P^x \in \Syl_p(PN)$, so $P^x = P^n$ for some $n \in N$. Now, $xn^{-1} \in N_G(P)$, and it follows that $xN \in N_G(P)N/N$, as desired.
		
		Let $N$ be a nontrivial $p$-subgroup of $Z(G)$. For the second statement, the induction hypothesis and the previous paragraph imply then that $G/N$ has a normal $p$-complement, say $K/N$. By the Schur-Zassenhaus Theorem, $N$ has a complement in $K$, say $L$, which is normal in $K$ since $N$ centralizes $L$. Thus, $L$ is also a characteristic subgroup of $K$. $L$ is then a normal subgroup in $G$ with index $|P|$, so $L$ is a normal $p$-complement, a contradiction.
	\end{proof}

	\section{The Proof of the Theorem}

	If $G$ is a minimal counterexample, it is possible to choose a nontrivial character $\zeta$ from the principal block of $G$. Moreover, we can choose such a character having height zero. Assume from now on that $\zeta$ is such a character. In particular, $\zeta$ cannot vanish on any $p$-element by Lemma \ref{2:4}. 
	
	Lemma \ref{2:6} implies that $C_G(z)$ is proper for each nonidentity $p$-element, and so, each such $C_G(z)$ has a normal $p$-complement by induction. From Lemmas \ref{2:2} and \ref{2:3} with $\zeta = \psi = \eta$, we obtain
	
	\begin{equation}\label{E:3}
		|G| = |G|\la \zeta , \zeta \ra_G \geq |G:P|\sum_{z \in P^{\#}} |\zeta(z)|^2.
	\end{equation}
	Now, the average over Galois conjugates of a totally positive algebraic integer is at least 1 as a consequence of the AM-GM inequality. The sum in \eqref{E:3} is a sum of nonzero, totally positive algebraic integers, and the terms are permuted by the Galois action. Thus, the average among all terms is at least 1, and
	
	\begin{equation}\label{E:4}
		\sum_{z \in P^{\#}} |\zeta(z)|^2 \geq |P| - 1.
	\end{equation}
	
	Since a character value has absolute value 1 iff it is a root of unity by a theorem of Kronecker, Siegel's theorem implies that equality holds in \eqref{E:4} precisely when $\zeta(z)$ is a root of unity for all $z \in P^{\#}$. Otherwise, if the inequality in \eqref{E:4} is strict, the sum must be at least $|P|$ since this sum is a rational integer. If this were the case, Lemma \ref{2:3} then implies that $\zeta$ vanishes on $G^\circ$. In particular, $\zeta(1) = 0$, an impossibility. Thus, $\zeta$ takes root of unity values on $P^{\#}$, and
	
	$$
	\sum_{z \in P^{\#}} |\zeta(z)|^2 = |P| - 1.
	$$
	Combining the above equation with Lemma \ref{2:3} gives
	
	\begin{equation}\label{E:6}
		|G:P| = \sum_{g \in G^\circ} |\zeta(g)|^2.
	\end{equation}
	We now consider $\la \zeta , 1_G \ra_G$ for which Lemma \ref{2:3} yields
	
	$$
	0 = |G| \la \zeta , 1_G \ra_G = \sum_{g \in G^\circ} \zeta(g) + |G:P|\sum_{z \in P^{\#}} \zeta(z).
	$$
	Set $a = \sum_{z \in P^{\#}}\zeta(z)$, and notice that $a$ is a rational integer. Also, $a \neq 0$ by Lemma \ref{2:4} since $a$ is expressed as a sum of $|P|$th roots of unity having $(|P| - 1)\zeta(1)$ terms. We may rewrite the above equation as
	
	\begin{equation}\label{E:7}
		-|G:P|a = \sum_{g \in G^\circ}\zeta(g).
	\end{equation}
	Also, by considering $\la 1_G , 1_G \ra_G$, Lemma \ref{2:3} yields
	
	$$
	|G| = |G^\circ| + |G:P|(|P| - 1),
	$$
	so in particular,
	\begin{equation}\label{E:8}
		|G:P| = |G^\circ|.
	\end{equation}
	
	We will use \eqref{E:6}, \eqref{E:7}, and \eqref{E:8} to derive a contradiction. The triangle inequality along with \eqref{E:7} and \eqref{E:8} yields
	\begin{equation}\label{E:9}
		\sum_{g \in G^\circ} |\zeta(g)| \geq \big|\sum_{g \in G^\circ} \zeta(g)\big| = |G^\circ||a| \geq |G^\circ|.
	\end{equation}
	Also, as a consequence of \eqref{E:6}, \eqref{E:8}, and the Cauchy--Schwarz inequality (see \cite[Lemma 4.10]{Is} for our formulation), we have that
	
	\begin{equation}\label{E:10}
		|G^\circ| = \sum_{g \in G^\circ} |\zeta(g)|^2 \geq \dfrac{1}{|G^\circ|}\bigg(\sum_{g \in G^\circ} |\zeta(g)|\bigg)^2 \geq \sum_{g \in G^\circ} |\zeta(g)|
	\end{equation}
	The rightmost inequality in \eqref{E:10} holds since $\sum_{g \in G^\circ} |\zeta(g)| \geq |G^\circ|$ by \eqref{E:9}. Combining \eqref{E:9} and \eqref{E:10} yields that equality holds throughout. Namely, we obtain $|a| = 1$ and
	
	$$
	\sum_{g \in G^\circ} |\zeta(g)| = \big|\sum_{g \in G^\circ} \zeta(g)\big|.
	$$
	Since $\zeta(1) > 0$, the triangle inequality allows us to conclude that $\zeta(g)$ is real and nonnegative for all $g \in G^\circ$.  Comparing with \eqref{E:7}, we see that $a = -1$. 
	
	Using \eqref{E:6}, \eqref{E:7}, \eqref{E:8}, and Lemma \ref{2:3}, we have
	
	$$
	\sum_{g \in G^\circ}|\zeta(g) - 1|^2 = \sum_{g \in G^\circ} \bigg[|\zeta(g)|^2 -\zeta(g) - \ol{\zeta(g)} + 1\bigg] 
	$$ $$
	= |G^\circ| -|G^\circ| - |G^\circ| + |G^\circ| = 0.
	$$
	This implies that $\zeta(g) = 1$ for all $g \in G^\circ$, so $\zeta$ is a linear character containing $G^\circ$ in its kernel, a contradiction via Lemma \ref{2:5}.

	
\end{document}